\documentclass[pdflatex,sn-mathphys-num]{sn-jnl}

\usepackage[T1]{fontenc}
\usepackage[utf8]{inputenc}
\usepackage{amsmath,amssymb,amsfonts,mathtools}
\usepackage{mathrsfs}
\usepackage{enumitem}
\usepackage{microtype}

\numberwithin{equation}{section}

\theoremstyle{thmstyleone}
\newtheorem{theorem}{Theorem}[section]
\newtheorem{assumption}[theorem]{Assumption}
\newtheorem{proposition}[theorem]{Proposition}
\newtheorem{lemma}[theorem]{Lemma}
\newtheorem{corollary}[theorem]{Corollary}

\theoremstyle{thmstylethree}
\newtheorem{definition}[theorem]{Definition}
\newtheorem{remark}[theorem]{Remark}

\DeclareMathOperator*{\esssup}{ess\,sup}

\newcommand{\R}{\mathbb R}

\newcommand{\eps}{\varepsilon}
\newcommand{\loc}{\mathrm{loc}}
\newcommand{\reg}{\mathrm{reg}}
\newcommand{\CKN}{\mathrm{CKN}}

\newcommand{\dx}{\,dx}
\newcommand{\dt}{\,dt}
\newcommand{\dxdt}{\,dx\,dt}
\newcommand{\dxds}{\,dx\,ds}
\newcommand{\nabh}{\nabla_h}

\newcommand{\Ph}{\mathbb P_h}

\newcommand{\norm}[2]{\left\|#1\right\|_{#2}}
\newcommand{\calH}{\mathcal H}
\newcommand{\calL}{\mathcal L}
\newcommand{\calE}{\mathcal E}
\newcommand{\calX}{\mathcal X}
\newcommand{\calY}{\mathcal Y}
\newcommand{\calZ}{\mathcal Z}

\newcommand{\rem}{\mathrm{rem}}
\newcommand{\harm}{\mathrm{harm}}
\newcommand{\locp}{\mathrm{loc}}
\newcommand{\rel}{\mathrm{rel}}

\title[Finite-scale one-component regularity]{Finite-Scale One-Component Regularity via Harmonic Pressure for the 3D Navier--Stokes Equations}

\author*[]{\fnm{Runlong} \sur{Yu}}\email{ryu5@ua.edu}

\affil[]{ \orgname{The University of Alabama}, \orgaddress{\city{Tuscaloosa}, \state{AL}, \country{USA}}}

\abstract{We study a finite-scale one-component regularity mechanism for suitable weak solutions of the three-dimensional incompressible Navier--Stokes equations.  The results are organized in three layers.  The first layer is unconditional.  Under a fixed scale-invariant local bound
\[
        \Phi(1)=A(1)+E(1)+C(1)+D(1)\le M,
\]
smallness of the critical vertical-component quantity
\[
        C_3(1)=\int_{Q_1}|u_3|^3\,dx\,dt
\]
yields a positive lower bound, depending only on \(M\), for the local regularity radius at the origin.  The proof converts one-component smallness into approximation by the two-and-a-half-dimensional limiting class and then into Caffarelli--Kohn--Nirenberg smallness at a smaller scale.  The pressure approximation is measured in a quotient by spatially harmonic functions.  This pressure topology reflects a genuine obstruction: time-dependent harmonic pressures may have bounded scale-invariant \(L^{3/2}\)-oscillation while their pointwise gradients lie beyond the control provided by the available scale-invariant quantities.
The second layer is a conditional logarithmic refinement.  A prepared two-shadow comparison package replaces the abstract compactness modulus by a logarithmic modulus and gives a logarithmic finite-scale decay.  The third layer is a conditional relaxed-shadowing refinement.  The comparison class is enlarged to smooth no-stretching horizontal flows \(V=(v_h,0)\), with the comparison pressure allowed to have \(\partial_3\pi\ne0\).  The resulting vertical residual pairs with the small component \(u_3\) in the relative-energy identity.  Under the buffered strong-flow and localized relaxed stability inputs stated below, this gives a power-type relaxed harmonic approximation and a power-type finite-scale decay.  The unconditional theorem is separated from the logarithmic and power-type assumptions; the latter two layers identify the quantitative stability mechanisms needed to upgrade the compactness modulus.}

\keywords{Navier--Stokes equations, suitable weak solutions, partial regularity}

\pacs[MSC Classification]{35Q30, 35B65, 35B45, 76D05}

\begin{document}
\maketitle

\section{Introduction}

The local regularity theory for the three-dimensional incompressible Navier--Stokes equations is built on scale-invariant estimates, pressure decompositions, compactness, and epsilon-regularity.  In the unit parabolic cylinder
\[
        Q_1=B_1(0)\times(-1,0)\subset \R^3\times\R,
\]
we consider
\begin{equation}
        \partial_tu-\Delta u+(u\cdot\nabla)u+\nabla p=0,
        \qquad \nabla\cdot u=0.
        \label{eq:NS}
\end{equation}
The Navier--Stokes scaling is
\[
        u_\lambda(x,t)=\lambda u(\lambda x,\lambda^2t),
        \qquad
        p_\lambda(x,t)=\lambda^2p(\lambda x,\lambda^2t).
\]
For \(0<r\le 1\), define
\[
\begin{aligned}
A(r)&=\esssup_{-r^2<t<0}\frac1r\int_{B_r}|u(x,t)|^2\dx,\\
E(r)&=\frac1r\int_{Q_r}|\nabla u|^2\dxdt,
\qquad
C(r)=\frac1{r^2}\int_{Q_r}|u|^3\dxdt,\\
D(r)&=\frac1{r^2}\int_{Q_r}|p-(p)_{B_r}(t)|^{3/2}\dxdt,
\qquad
C_3(r)=\frac1{r^2}\int_{Q_r}|u_3|^3\dxdt,
\end{aligned}
\]
and set
\[
        \Phi(r)=A(r)+E(r)+C(r)+D(r),
        \qquad
        \Psi(r)=C(r)+D(r).
\]
These quantities are invariant under the Navier--Stokes scaling.
The weak-solution framework goes back to Leray and Hopf
\cite{Leray1934,Hopf1951}; the classical theory is developed, for
example, in \cite{Seregin2015}.  The Prodi--Serrin
line of criteria gives an early scale-critical conditional regularity
mechanism \cite{Prodi1959,Serrin1962,KozonoSohr1997}.  The partial
regularity theorem of Caffarelli, Kohn and Nirenberg \cite{CKN1982},
building on Scheffer's work \cite{Scheffer1976,Scheffer1977}, asserts
that suitable scale-invariant smallness implies local regularity; further
work on the structure of singular sets includes \cite{ChoeLewis2000}.
Refinements of local energy, pressure, endpoint, Morrey-space, and
concentration criteria include
\cite{Lin1998,LadyzhenskayaSeregin1999,Struwe1988,SohrWahl1986,Seregin2007Morrey,Seregin2007Local,Vasseur2007,GustafsonKangTsai2007,SereginSverak2002,EscauriazaSereginSverak2003,JiaSverak2014,GuevaraPhuc2017,Wolf2017,BarkerPrange2021,AlbrittonBarkerPrange2023}.

A second line of work studies criteria involving fewer components of the
unknowns.  Regularity criteria involving one velocity component, one
velocity derivative, one vorticity component, or one preferred direction
appear in
\cite{PenelPokorny2004,KukavicaZiane2006,KukavicaZiane2007,ZhouPokorny2009,CaoTiti2011,CheminZhang2016,CheminZhangZhang2017,KukavicaRusinZiane2017,HanLeiLiZhao2019,KangNguyen2023}.  The guiding intuition is clear: when one component is
small at a critical scale, the flow should approach a regime with
partially lower-dimensional structure.  For the vertical component
\(u_3\), the formal limiting velocity has the form
\[
        v=(v_h,0),\qquad \nabh\cdot v_h=0.
\]
The limiting equations retain horizontal transport and full three-dimensional diffusion:
\begin{equation}
\begin{cases}
\partial_t v_h-\Delta v_h+(v_h\cdot\nabh)v_h+\nabh q=0,\\
\nabh\cdot v_h=0,\\
\partial_3q=0.
\end{cases}
\label{eq:limiting-system-intro}
\end{equation}
We refer to \eqref{eq:limiting-system-intro} as the two-and-a-half-dimensional limiting system.  The horizontal velocity may still depend on \(x_3\), and the Laplacian remains the full operator \(\Delta=\Delta_h+\partial_3^2\).

The preceding intuition hides several analytic difficulties.  Smallness of \(u_3\) leaves the horizontal velocity \(u_h\) largely uncontrolled.  The nonlinear term \((u_h\cdot\nabh)u_h\) may remain large, and the Caffarelli--Kohn--Nirenberg criterion requires smallness of the combined velocity-pressure quantity \(\Psi=C+D\).  Thus the main task is to show that one-component smallness forces, at a smaller scale, a full CKN quantity to fall below the universal threshold.

The pressure creates a second and more delicate issue.  Locally, the pressure decomposes into a Calderon--Zygmund part determined by the quadratic velocity terms and a spatially harmonic part.  The Calderon--Zygmund part is compact under strong \(L^3\)-convergence of velocities.  The harmonic part may oscillate sharply in time while preserving bounded scale-invariant \(L^{3/2}\)-oscillation.  Consequently, the natural approximation topology is the quotient
\[
        p\approx q+h,
        \qquad \Delta h(\cdot,t)=0.
\]
This quotient is forced by a simple obstruction: within the limiting class \eqref{eq:limiting-system-intro}, the full pressure gradient can be arbitrarily large under a fixed bound on \(\Phi\).  The proof below makes this obstruction explicit and then uses harmonic oscillation decay to recover the pressure contribution needed for CKN smallness.

A third difficulty concerns rates.  Compactness gives an approximation modulus \(\omega_{M,\theta}(s)\to0\) as \(s\downarrow0\), leaving the dependence on \(C_3(1)\) implicit.  Logarithmic and power-type rates require quantitative shadowing estimates, local pressure reconstruction, and careful treatment of cutoff errors in the relative-energy argument.  These estimates form separate stability problems.  This paper separates the unconditional compactness theorem from the conditional quantitative mechanisms that would replace the compactness modulus by explicit logarithmic or power-type rates.

The results are organized into three layers.  The first layer is an unconditional finite-scale theorem.  Under the fixed a priori bound \(\Phi(1)\le M\), there exist \(\eps_*(M)>0\) and \(\rho_*(M)>0\) such that
\[
        C_3(1)\le \eps_*(M)
        \quad\Longrightarrow\quad
        r_{\reg}(0,0)\ge \rho_*(M).
\]
The constants are qualitative, since they come from compactness, while the conclusion is finite-scale because the radius depends only on \(M\).

The second layer is a conditional logarithmic refinement.  We formulate a prepared comparison package which, once available, gives
\[
        \calX^{\harm}_{\theta/4}(u,p;M)
        \le C_{M,\theta}|\log C_3(1)|^{-\sigma}.
\]
The main stability mechanism is a two-shadow estimate.  A rough limiting shadow propagates the smoothing error, and a smoothed limiting shadow supplies the comparison object.  This arrangement keeps the smoothing error outside the large Gronwall factor generated by the smoothed flow.

The third layer is a conditional relaxed-shadowing route toward a power-type estimate.  The comparison class is enlarged to smooth horizontal no-stretching flows \(V=(v_h,0)\) satisfying
\[
        \partial_t v_h-\Delta v_h+(v_h\cdot\nabh)v_h+\nabh\pi=0,
        \qquad \nabh\cdot v_h=0,
\]
with \(\partial_3\pi\) allowed.  In the full three-dimensional equation this leaves a vertical residual \((0,0,\partial_3\pi)\), and the relative-energy identity pairs this residual with the small component \(u_3\).  Together with harmonic pressure reconstruction, this gives a conditional power-type finite-scale decay.

The paper is organized as follows.  Section~\ref{sec:main-results} states the layered results.  Section~\ref{sec:preliminaries} records suitable weak solutions, scale-invariant quantities, epsilon regularity, interpolation, and pressure estimates.  Section~\ref{sec:limiting} studies the limiting system, proves the pressure-gradient obstruction, and obtains decay for the limiting class.  Section~\ref{sec:compactness} proves the harmonic-pressure compactness approximation.  Section~\ref{sec:unconditional-decay} proves the unconditional finite-scale decay and regularity-radius theorem.  Section~\ref{sec:log} develops the conditional logarithmic refinement.  Section~\ref{sec:relaxed} develops the conditional relaxed-shadowing power-type refinement.  Section~\ref{sec:remarks} records limitations and possible further refinements.

\section{Main results}
\label{sec:main-results}

We state the results in three layers.  The first one is unconditional and is the main theorem.  The second and third are conditional refinements which identify additional quantitative inputs.

\subsection{Unconditional finite-scale theorem}

\begin{theorem}[Finite-scale one-component regularity]
\label{thm:finite-scale-main}
For every \(M\ge1\), there exist constants
\[
        \eps_*(M)>0,
        \qquad
        \rho_*(M)>0,
\]
such that the following holds.  Let \((u,p)\) be a suitable weak solution of \eqref{eq:NS} in \(Q_1\).  If
\[
        \Phi(1)\le M,
        \qquad
        C_3(1)\le \eps_*(M),
\]
then \(u\) is regular in \(Q_{\rho_*(M)}\).  In particular,
\[
        r_{\reg}(0,0)\ge \rho_*(M).
\]
\end{theorem}

The proof gives a decay estimate with a compactness modulus.

\begin{theorem}[Finite-scale decay with compactness modulus]
\label{thm:compactness-decay}
Let \(M\ge1\) and \(0<\theta<1/2\).  There exist a nondecreasing modulus
\[
        \omega_{M,\theta}:[0,\infty)\to[0,\infty),
        \qquad
        \lim_{s\downarrow0}\omega_{M,\theta}(s)=0,
\]
and constants \(K(M,\theta)\ge1\), \(r_0(M,\theta)\in(0,\theta)\), such that every suitable weak solution in \(Q_1\) satisfying \(\Phi(1)\le M\) and \(C_3(1)=\delta\) obeys
\begin{equation}
        \Psi(r)
        \le K(M,\theta)r+K(M,\theta)r^{-2}\omega_{M,\theta}(\delta)
        \label{eq:compact-decay-main}
\end{equation}
for every \(0<r<r_0(M,\theta)\).
\end{theorem}

The contrapositive gives the corresponding concentration statement.

\begin{corollary}[One-component concentration near singular points]
\label{cor:concentration}
Let \(M\ge1\), and let \(\eps_*(M)\) be as in Theorem~\ref{thm:finite-scale-main}.  If \((u,p)\) is suitable in \(Q_1\), \(\Phi(1)\le M\), and \((0,0)\) is singular, then
\[
        C_3(1)>\eps_*(M).
\]
More generally, if \(z_0=(x_0,t_0)\) is singular and \((u,p)\) is suitable in \(Q_R(z_0)\), then for every \(0<r\le R\) such that \(\Phi(z_0,r)\le M\), one has
\[
        C_3(z_0,r)>\eps_*(M).
\]
\end{corollary}

\subsection{Conditional logarithmic layer}

The logarithmic layer replaces the abstract compactness modulus by a logarithmic modulus under a prepared comparison estimate.  The precise assumption is stated in Section~\ref{sec:log}.  It gives, for a smoothing parameter \(\ell\), an estimate of the form
\[
        \calE^{\harm}_{\theta/4}
        \lesssim
        \ell^a+\ell^{-N}\delta^b+\exp(C\ell^{-N})\delta^b.
\]
Optimizing \(\ell\) gives the following result.

\begin{theorem}[Conditional logarithmic harmonic-pressure approximation]
\label{thm:log-main}
Assume the prepared comparison estimate, Assumption~\ref{ass:prepared-comparison}.  Let \(M\ge1\) and \(0<\theta<1/2\).  Then there exist constants
\[
        C_{M,\theta}\ge1,
        \qquad
        \sigma>0,
        \qquad
        \delta_{M,\theta}\in(0,1),
\]
such that every suitable weak solution in \(Q_1\) satisfying \(\Phi(1)\le M\) and \(\delta=C_3(1)\le\delta_{M,\theta}\) obeys
\begin{equation}
        \calX^{\harm}_{\theta/4}(u,p;M)
        \le C_{M,\theta}|\log\delta|^{-\sigma}.
        \label{eq:log-main-excess}
\end{equation}
Consequently,
\begin{equation}
        \Psi(r)
        \le C_{M,\theta}r+C_{M,\theta}r^{-2}|\log\delta|^{-\sigma}
        \label{eq:log-decay-main}
\end{equation}
for all sufficiently small \(r\), and after decreasing \(\delta_{M,\theta}\),
\begin{equation}
        r_{\reg}(0,0)
        \ge c_{M,\theta}|\log\delta|^{-\sigma/3}.
        \label{eq:log-radius-main}
\end{equation}
\end{theorem}

\subsection{Conditional relaxed-shadowing power layer}

The power layer uses a larger comparison class.  It is conditional on two inputs stated in Section~\ref{sec:relaxed}: a buffered strong bound for the relaxed no-stretching flow and a localized relaxed weak--strong stability estimate.  No unconditional power-rate conclusion is claimed in this subsection; Theorem~\ref{thm:power-main} is only a consequence of Assumptions~\ref{ass:relaxed-strong} and \ref{ass:relaxed-stability}.

\begin{theorem}[Conditional relaxed-shadowing power-type decay]
\label{thm:power-main}
Assume Assumptions~\ref{ass:relaxed-strong} and \ref{ass:relaxed-stability}.  Let \(M\ge1\) and fix \(0<\theta<1/16\).  Then there exist constants
\[
        C_H(M,\theta)\ge1,
        \qquad
        \alpha>0,
        \qquad
        \Gamma>0,
        \qquad
        r_H(M,\theta)>0,
\]
such that every suitable weak solution in \(Q_1\) satisfying
\[
        \Phi(1)\le M,
        \qquad
        C_3(1)=\delta\le1,
\]
obeys
\begin{equation}
        \Psi(r)
        \le C_H(M,\theta)r^{\alpha}+C_H(M,\theta)r^{-2}\delta^{\Gamma}
        \label{eq:power-decay-main}
\end{equation}
for every \(0<r<r_H(M,\theta)\).  Consequently, after decreasing the admissible range of \(\delta\),
\begin{equation}
        r_{\reg}(0,0)
        \ge c_H(M,\theta)\delta^{\Gamma/(\alpha+2)}.
        \label{eq:power-radius-main}
\end{equation}
The same decay estimate also yields a fixed-radius statement: there are \(\delta_H(M,\theta)>0\) and \(\rho_H(M,\theta)>0\) such that \(C_3(1)\le\delta_H(M,\theta)\) implies \(r_{\reg}(0,0)\ge\rho_H(M,\theta)\).
\end{theorem}

\begin{remark}
Theorems~\ref{thm:log-main} and \ref{thm:power-main} are stronger than Theorem~\ref{thm:finite-scale-main} in rate information, only in their rate information; the logical scope of the unconditional theorem is unchanged.  Theorem~\ref{thm:finite-scale-main} is unconditional.  The logarithmic and power-type theorems are conditional statements that isolate the additional quantitative stability mechanisms needed to replace the compactness modulus.
\end{remark}

\section{Preliminaries}
\label{sec:preliminaries}

For \(z_0=(x_0,t_0)\in\R^3\times\R\), define
\[
        B_r(x_0)=\{x\in\R^3:|x-x_0|<r\},
        \qquad
        Q_r(z_0)=B_r(x_0)\times(t_0-r^2,t_0).
\]
When \(z_0=(0,0)\), write \(B_r=B_r(0)\) and \(Q_r=Q_r(0,0)\).  The spatial average is
\[
        (f)_{B_r(x_0)}(t)=\frac1{|B_r|}\int_{B_r(x_0)}f(x,t)\dx.
\]
For general centers, the quantities \(A,E,C,D,C_3,\Phi,\Psi\) are defined by translation and scaling.

\begin{definition}[Suitable weak solution]
A pair \((u,p)\) is a suitable weak solution of \eqref{eq:NS} in a parabolic cylinder \(Q\) if
\[
        u\in L^\infty_tL^2_x(Q)\cap L^2_tH^1_x(Q),
        \qquad
        p\in L^{3/2}(Q),
\]
\((u,p)\) solves \eqref{eq:NS} in distributions, \(\nabla\cdot u=0\), and the local energy inequality holds: for every nonnegative \(\phi\in C_c^\infty(Q)\) and a.e. \(t\),
\begin{equation}
\begin{aligned}
\int |u(x,t)|^2\phi(x,t)\dx
&+2\int_{-\infty}^t\int |\nabla u|^2\phi\dxds \\
&\le
\int_{-\infty}^t\int |u|^2(\partial_s\phi+\Delta\phi)\dxds
+\int_{-\infty}^t\int (|u|^2+2p)u\cdot\nabla\phi\dxds.
\end{aligned}
\label{eq:LEI}
\end{equation}
\end{definition}

The pressure is defined up to a function of time.  All pressure quantities below are invariant under this change.

We use
\[
        r_{\reg}(z_0)=\sup\{r>0:u\in L^\infty(Q_r(z_0))\}.
\]
Only positivity and universal changes of scale are used.

\begin{theorem}[CKN epsilon regularity]
\label{thm:CKN}
There exist universal constants \(\eps_{\CKN}>0\), \(\kappa\in(0,1/2)\), and \(C_{\CKN}<\infty\) such that the following holds.  If \((u,p)\) is suitable in \(Q_r(z_0)\) and
\[
        \Psi(z_0,r)=C(z_0,r)+D(z_0,r)\le\eps_{\CKN},
\]
then \(u\) is regular in \(Q_{\kappa r}(z_0)\), and
\[
        \esssup_{Q_{\kappa r}(z_0)}|u|
        \le C_{\CKN}r^{-1}.
\]
In particular, \(r_{\reg}(z_0)\ge \kappa r\).
\end{theorem}

\begin{lemma}[Scale-invariant interpolation]
\label{lem:interp}
For every suitable weak solution and every \(0<r\le1\),
\[
        C(r)\le C A(r)^{3/4}E(r)^{3/4}+C A(r)^{3/2}.
\]
\end{lemma}

\begin{proof}
For a.e. \(t\), the Gagliardo--Nirenberg inequality on \(B_r\) gives
\[
\norm{u(\cdot,t)}{L^3(B_r)}^3
\le C\norm{u(\cdot,t)}{L^2(B_r)}^{3/2}
       \norm{\nabla u(\cdot,t)}{L^2(B_r)}^{3/2}
+Cr^{-3/2}\norm{u(\cdot,t)}{L^2(B_r)}^3.
\]
Integrating in time, using Holder's inequality, and dividing by \(r^2\) gives the claim.
\end{proof}

\begin{lemma}[Fixed-scale nesting of pressure oscillations]
\label{lem:pressure-nesting}
Let \(0<r<R\).  Then, for a.e. time,
\[
        \int_{B_r}|p-(p)_{B_r}(t)|^{3/2}\dx
        \le C\int_{B_R}|p-(p)_{B_R}(t)|^{3/2}\dx.
\]
Consequently, whenever \(Q_r\subset Q_R\),
\[
        D(r)\le C\left(\frac Rr\right)^2D(R).
\]
\end{lemma}

\begin{proof}
For \(p_0=3/2\), the mean minimizes the \(L^{p_0}\)-distance up to a universal constant.  Hence for any constant \(c\),
\[
        \int_{B_r}|p-(p)_{B_r}|^{p_0}\dx
        \le C\int_{B_r}|p-c|^{p_0}\dx.
\]
Choose \(c=(p)_{B_R}\) and use \(B_r\subset B_R\).  Integrating over the shorter time interval and multiplying by \(r^{-2}\) gives the scale-invariant estimate.
\end{proof}

For a pressure-like scalar \(f\), set
\[
        D_f(r)=\frac1{r^2}\int_{Q_r}|f-(f)_{B_r}(t)|^{3/2}\dxdt.
\]

\begin{lemma}[Harmonic corrector decay]
\label{lem:harmonic-decay}
Let \(0<r\le R/2\), and let \(h\in L^{3/2}(Q_R)\) satisfy
\[
        \Delta h(\cdot,t)=0\quad\text{in }B_R
\]
for a.e. \(t\in(-R^2,0)\).  Then
\[
        D_h(r)\le C\left(\frac rR\right)^{5/2}D_h(R).
\]
In particular, if \(h\) is harmonic in space in \(Q_\theta\) and \(D_h(\theta)\le B\), then for \(0<r\le\theta/2\),
\[
        D_h(r)\le C\theta^{-5/2}Br^{5/2}\le C\theta^{-5/2}Br.
\]
\end{lemma}

\begin{proof}
For a.e. \(t\), harmonic oscillation estimates give
\[
        \int_{B_r}|h-(h)_{B_r}(t)|^{3/2}\dx
        \le C\left(\frac rR\right)^{9/2}
        \int_{B_R}|h-(h)_{B_R}(t)|^{3/2}\dx.
\]
Integrate over \((-r^2,0)\), enlarge the time interval on the right to \((-R^2,0)\), and multiply by \(r^{-2}\).
\end{proof}

\begin{lemma}[Pressure decomposition and decay]
\label{lem:pressure-decomp}
There is a universal constant \(C\) such that, whenever \(0<r\le R/2\),
\begin{equation}
        D(r)\le C\frac rR D(R)+C\left(\frac Rr\right)^2C(R).
        \label{eq:pressure-decay}
\end{equation}
More precisely, if in \(B_{3R/4}\) one writes
\[
        p=p_{\locp}+p_{\harm},
        \qquad
        p_{\locp}=R_iR_j(\chi u_i u_j),
\]
where \(\chi\in C_c^\infty(B_R)\) and \(\chi\equiv1\) on \(B_{3R/4}\), then \(p_{\harm}\) is harmonic in space and
\[
        D_{p_{\harm}}(r)
        \le C\left(\frac rR\right)^{5/2}D_{p_{\harm}}(R),
\]
while
\[
        \frac1{r^2}\int_{Q_r}|p_{\locp}|^{3/2}\dxdt
        \le C\left(\frac Rr\right)^2 C(R).
\]
\end{lemma}

\begin{proof}
The local part is controlled by Calderon--Zygmund estimates:
\[
        \norm{p_{\locp}(\cdot,t)}{L^{3/2}(B_R)}
        \le C\norm{u(\cdot,t)}{L^3(B_R)}^2.
\]
The harmonic part is controlled by Lemma~\ref{lem:harmonic-decay}.  Combining the two parts and using Lemma~\ref{lem:pressure-nesting} gives \eqref{eq:pressure-decay}.
\end{proof}

\section{The limiting system and the harmonic pressure obstruction}
\label{sec:limiting}

Writing \(u=(u_h,u_3)\), the Navier--Stokes equations become
\begin{align}
\partial_tu_h-\Delta u_h+(u_h\cdot\nabh)u_h+u_3\partial_3u_h+\nabh p&=0,
\label{eq:horizontal}\\
\partial_tu_3-\Delta u_3+u_h\cdot\nabh u_3+u_3\partial_3u_3+\partial_3p&=0,
\label{eq:vertical}
\end{align}
and
\[
        \nabh\cdot u_h+\partial_3u_3=0.
\]
If \(u_3^{(n)}\to0\) strongly in \(L^3_{\loc}\) and \(u^{(n)}\to v\) strongly locally, then \(v=(v_h,0)\), \(\nabh\cdot v_h=0\), and the limiting equations are
\begin{equation}
\begin{cases}
\partial_t v_h-\Delta v_h+(v_h\cdot\nabh)v_h+\nabh q=0,\\
\nabh\cdot v_h=0,\\
\partial_3q=0.
\end{cases}
\label{eq:limiting-system}
\end{equation}
The horizontal velocity may depend on \(x_3\), and the full Laplacian remains present.

\begin{proposition}[Absence of a full pressure-gradient bound]
\label{prop:pressure-obstruction}
There exists a universal number \(M_*\ge1\) with the following property.  For every \(L>0\), there is a smooth solution \((v,q)\) of \eqref{eq:limiting-system} in \(Q_{1/2}\) such that
\[
        \Phi_v(1/2)\le M_*,
        \qquad
        \norm{\nabla q}{L^\infty(Q_{1/4})}>L.
\]
Consequently, every proposed estimate of the form
\[
        \norm{\nabla q}{L^\infty(Q_{1/4})}
        \le K(\Phi_v(1/2))
\]
fails for the full limiting class.
\end{proposition}

\begin{proof}
Choose \(\varphi\in C_c^\infty((-1,1))\) with \(\varphi'\not\equiv0\), and choose \(t_0\in(-1/16,-1/32)\).  For \(H\gg1\), set
\[
        \eps_H=H^{-1/2},
        \qquad
        \delta_H=H^{-3/2},
        \qquad
        a_H(t)=\eps_H\varphi\left(\frac{t-t_0}{\delta_H}\right).
\]
For \(H\) sufficiently large, the support of \(a_H\) is contained in \((-1/16,-1/32)\).  Define
\[
        v_H(x,t)=(a_H(t),0,0),
        \qquad
        q_H(x,t)=-a_H'(t)x_1.
\]
Then \(v_{H,3}=0\), \(\nabh\cdot v_{H,h}=0\), \(\partial_3q_H=0\), and
\[
        \partial_t v_{H,h}+\nabh q_H=0,
        \qquad
        \Delta v_{H,h}=0,
        \qquad
        (v_{H,h}\cdot\nabh)v_{H,h}=0.
\]
Thus \((v_H,q_H)\) solves \eqref{eq:limiting-system}.  Since \(v_H\) is spatially constant,
\[
        A_{v_H}(1/2)\lesssim\eps_H^2,
        \quad
        E_{v_H}(1/2)=0,
        \quad
        C_{v_H}(1/2)\lesssim\eps_H^3.
\]
Moreover \((q_H)_{B_r}(t)=0\), and
\[
        \int |a_H'(t)|^{3/2}\dt
        =\eps_H^{3/2}\delta_H^{-1/2}\int |\varphi'(s)|^{3/2}\,ds
        \sim 1.
\]
Thus \(D_{q_H}(1/2)\) is uniformly bounded.  On the other hand,
\[
        \norm{\nabla q_H}{L^\infty(Q_{1/4})}
        =\norm{a_H'}{L^\infty}
        \sim \frac{\eps_H}{\delta_H}=H\to\infty.
\]
The proposition follows.
\end{proof}

\begin{remark}
The obstruction is harmonic in space: \(q_H(x,t)=-a_H'(t)x_1\) satisfies \(\Delta q_H=0\).  Its scale-invariant \(L^{3/2}\)-oscillation is bounded while its pointwise gradient can be arbitrarily large because of time concentration.  This is why the approximation below is formulated modulo spatially harmonic pressures.
\end{remark}

A suitable solution of \eqref{eq:limiting-system} means a suitable weak solution of the full Navier--Stokes system whose velocity has the form \(v=(v_h,0)\) and whose pressure satisfies \(\partial_3q=0\).

We use the one-component criterion of Kang and Nguyen \cite{KangNguyen2023} in the following one-scale form.

\begin{lemma}[One-scale consequence of Kang--Nguyen]
\label{lem:KN}
Let \(\eps_W>0\) denote the velocity-only epsilon-regularity constant recalled in \cite{KangNguyen2023}.  For every \(M_0>0\) and every \(\eps_0\in(0,\eps_W]\), there exist
\[
        \eps_{KN}(M_0,\eps_0)>0,
        \qquad
        \delta_{KN}(M_0,\eps_0)\in(0,1),
\]
with the following property.  Let \((w,\pi)\) be a suitable weak solution in \(Q_{r_0}(z_0)\), and choose a pressure representative \(\pi\) such that
\[
        r_0^{-2}\int_{Q_{r_0}(z_0)}(|w|^3+|\pi|^{3/2})\dxdt\le M_0.
\]
If, for some \(1\le p,q\le\infty\),
\[
        r_0^{1-2/p-3/q}
        \norm{w_3}{L^p_tL^q_x(Q_{r_0}(z_0))}
        \le \eps_{KN}(M_0,\eps_0),
\]
then
\[
        (\delta_{KN}r_0)^{-2}
        \int_{Q_{\delta_{KN}r_0}(z_0)}|w|^3\dxdt
        \le \eps_0.
\]
Consequently, if \(w_3\equiv0\) and \(\eps_0=\eps_W\), then, after decreasing the radius by a universal factor,
\[
        w\in L^\infty(Q_{\delta_{KN}r_0/2}(z_0)),
        \qquad
        \norm{w}{L^\infty(Q_{\delta_{KN}r_0/2}(z_0))}
        \le C_W(\delta_{KN}r_0)^{-1}.
\]
\end{lemma}

\begin{proof}
This is the suitable-weak-solution case of \cite[Theorem 1.3]{KangNguyen2023}.  Subtracting a function of time from the pressure does not change either the distributional equation or the local energy inequality, because the velocity is divergence-free and the pressure appears only through the pairing with \(w\cdot\nabla\phi\).  Thus the chosen representative \(\pi\) is admissible.  The statement above is obtained from the theorem by translating and rescaling the cylinder; equivalently, one applies the local result on a bounded domain with \(Q_{r_0}(z_0)\) compactly contained in that domain.  The final assertion follows by applying the velocity-only epsilon-regularity criterion to \(Q_{\delta_{KN}r_0}(z_0)\).
\end{proof}

\begin{proposition}[Decay for the limiting class]
\label{prop:limiting-decay}
For every \(M\ge1\), there exist constants \(K_A(M)\ge1\) and \(r_A(M)\in(0,1/8]\) such that every suitable solution \((v,q)\) of \eqref{eq:limiting-system} in \(Q_{1/2}\) satisfying
\[
        \Phi_v(1/2)\le M
\]
obeys
\begin{equation}
        \Psi_v(r)=C_v(r)+D_q(r)\le K_A(M)r,
        \qquad 0<r\le r_A(M).
        \label{eq:limiting-decay}
\end{equation}
\end{proposition}

\begin{proof}
First obtain a velocity bound at a scale depending only on \(M\).  Replace the pressure by \(\pi=q-(q)_{B_{1/2}}(t)\).  Choose \(r_0=1/8\).  Since \(Q_{r_0}\subset Q_{1/2}\), the definitions and Lemma~\ref{lem:pressure-nesting} give
\[
        r_0^{-2}\int_{Q_{r_0}}(|v|^3+|\pi|^{3/2})\dxdt\le CM.
\]
Apply Lemma~\ref{lem:KN} with \(M_0=CM\) and \(\eps_0=\eps_W\).  Since \(v_3\equiv0\), the one-component smallness hypothesis is automatic.  Hence there are \(R=R(M)\in(0,1/8]\) and \(L=L(M)<\infty\) such that
\begin{equation}
        \norm{v}{L^\infty(Q_{2R})}\le L.
        \label{eq:limiting-Linfty}
\end{equation}
For \(0<r\le R\),
\[
        C_v(r)\le Cr^{-2}L^3|Q_r|\le CL^3r^3.
\]
We estimate the pressure while leaving the complete gradient \(\nabla q\) unestimated.  Choose \(\chi\in C_c^\infty(B_{2R})\) with \(\chi\equiv1\) on \(B_R\), and write
\[
        q=q_{\locp}+q_{\harm},
        \qquad
        q_{\locp}=R_aR_b(\chi v_av_b),
        \qquad a,b\in\{1,2\}.
\]
Then \(q_{\harm}\) is harmonic in space in \(B_R\) for a.e. time.  Since \(v\in L^\infty(Q_{2R})\), Calderon--Zygmund estimates imply, for every finite \(s\in(1,\infty)\),
\[
        \norm{q_{\locp}}{L^s(Q_R)}\le C_s(M,R).
\]
Taking \(s=4\) and using Holder's inequality gives, for \(0<r\le R/2\),
\[
\begin{aligned}
        D_{q_{\locp}}(r)
        &\le Cr^{-2}|Q_r|^{5/8}\norm{q_{\locp}}{L^4(Q_R)}^{3/2}  \\
        &\le C(M)r^{-2}(r^5)^{5/8}
        =C(M)r^{9/8}.
\end{aligned}
\]
Here the exponent is \(-2+25/8=9/8\), using the parabolic volume \(|Q_r|\sim r^5\).
For the harmonic part, Lemma~\ref{lem:harmonic-decay} gives
\[
        D_{q_{\harm}}(r)
        \le C\left(\frac rR\right)^{5/2}D_{q_{\harm}}(R).
\]
The fixed-scale factor is bounded by \(C(M)\), using Lemma~\ref{lem:pressure-nesting} and the bound on \(q_{\locp}\).  Thus
\[
        D_q(r)\le C(M)r^{9/8}+C(M)r^{5/2}.
\]
Together with the velocity estimate, this implies \eqref{eq:limiting-decay}, after decreasing the radius and increasing the constant.
\end{proof}

\section{Approximation modulo harmonic pressures}
\label{sec:compactness}

For \(0<\theta<1/2\), define
\[
        \calH(Q_\theta)=\{h\in L^{3/2}(Q_\theta):\Delta h(\cdot,t)=0\text{ in }B_\theta\text{ for a.e. }t\in(-\theta^2,0)\}.
\]
Let \(\calL_M(Q_\theta)\) be the class of suitable solutions \((v,q)\) of \eqref{eq:limiting-system} in \(Q_\theta\), with \(v=(v_h,0)\), \(\partial_3q=0\), and
\[
        \Phi_v(\theta)\le K_0(M,\theta),
\]
where \(K_0(M,\theta)\) is chosen large enough to contain all compactness limits of sequences satisfying \(\Phi(1)\le M\).

For \((u,p)\), \((v,q)\), and \(h\in\calH(Q_\theta)\), set
\[
\calE^{\harm}_\theta((u,p),(v,q);h)
=\frac1{\theta^2}\int_{Q_\theta}|u-v|^3\dxdt
+\frac1{\theta^2}\int_{Q_\theta}|p-q-h|^{3/2}\dxdt.
\]
Define the harmonic-pressure excess
\begin{equation}
\calX^{\harm}_\theta(u,p;M)
=\inf_{(v,q)\in\calL_M(Q_\theta)}\inf_{h\in\calH(Q_\theta)}
\calE^{\harm}_\theta((u,p),(v,q);h).
\label{eq:harmonic-excess}
\end{equation}

\begin{theorem}[Qualitative finite-scale approximation]
\label{thm:qualitative-approx}
Let \(M\ge1\) and \(0<\theta<1/2\).  There exists a nondecreasing modulus
\[
        \omega_{M,\theta}:[0,\infty)\to[0,\infty),
        \qquad
        \lim_{s\downarrow0}\omega_{M,\theta}(s)=0,
\]
such that every suitable weak solution \((u,p)\) of \eqref{eq:NS} in \(Q_1\) satisfying \(\Phi(1)\le M\) obeys
\begin{equation}
        \calX^{\harm}_\theta(u,p;M)
        \le \omega_{M,\theta}(C_3(1)).
        \label{eq:qualitative-approx}
\end{equation}
\end{theorem}

\begin{proof}
Suppose the conclusion fails.  Then there exist \(\eta_0>0\) and suitable weak solutions \((u^{(n)},p^{(n)})\) in \(Q_1\) such that
\[
        \Phi_{u^{(n)}}(1)\le M,
        \qquad
        C^{(n)}_3(1)\to0,
\]
and
\[
        \calX^{\harm}_\theta(u^{(n)},p^{(n)};M)\ge\eta_0
        \qquad\text{for all }n.
\]
The uniform energy bound gives, after passing to a subsequence,
\[
        u^{(n)}\rightharpoonup v\quad\text{weakly in }L^2_tH^1_x(Q_\sigma),
        \qquad
        u^{(n)}\stackrel{*}{\rightharpoonup} v\quad\text{in }L^\infty_tL^2_x(Q_\sigma)
\]
for every \(\sigma<1\).  After fixing pressure representatives by subtracting functions of time, the equation gives a uniform bound for \(\partial_tu^{(n)}\) in
\[
        L^{3/2}((-\sigma^2,0);W^{-1,3/2}(B_\sigma)).
\]
Aubin--Lions gives strong convergence in \(L^2(Q_\sigma)\).  Interpolation with the uniform \(L^{10/3}\)-bound yields
\begin{equation}
        u^{(n)}\to v\quad\text{strongly in }L^3_{\loc}(Q_1).
        \label{eq:strong-L3}
\end{equation}
Since \(u^{(n)}_3\to0\) strongly in \(L^3(Q_1)\), the limit has the form \(v=(v_h,0)\), and \(\nabh\cdot v_h=0\).

Set
\[
        \widetilde p^{(n)}=p^{(n)}-(p^{(n)})_{B_1}(t).
\]
After extraction,
\[
        \widetilde p^{(n)}\rightharpoonup q
        \quad\text{weakly in }L^{3/2}_{\loc}(Q_1).
\]
Fix \(\rho\) with \(\theta<\rho<1\), and choose \(\chi\in C_c^\infty(B_\rho)\) equal to one on a neighborhood of \(B_\theta\).  In \(Q_\theta\), decompose
\[
        \widetilde p^{(n)}=P^{(n)}+H^{(n)},
        \qquad
        P^{(n)}=R_iR_j(\chi u_i^{(n)}u_j^{(n)}),
\]
where \(H^{(n)}(\cdot,t)\) is harmonic in \(B_\theta\) for a.e. time.  By \eqref{eq:strong-L3},
\[
        P^{(n)}\to P:=R_iR_j(\chi v_iv_j)
        \quad\text{strongly in }L^{3/2}(Q_\theta).
\]
Because \(\widetilde p^{(n)}\rightharpoonup q\) and \(P^{(n)}\to P\) strongly, \(H^{(n)}=\widetilde p^{(n)}-P^{(n)}\) converges weakly in \(L^{3/2}(Q_\theta)\), after extraction, to \(H=q-P\).  The limit is harmonic in space.  Set
\[
        h^{(n)}=H^{(n)}-H+(p^{(n)})_{B_1}(t).
\]
Then \(h^{(n)}\in\calH(Q_\theta)\), and
\[
        p^{(n)}-q-h^{(n)}=P^{(n)}-P\to0
        \quad\text{strongly in }L^{3/2}(Q_\theta).
\]
Together with \eqref{eq:strong-L3}, this gives
\[
        \calE^{\harm}_\theta((u^{(n)},p^{(n)}),(v,q);h^{(n)})\to0.
\]
It remains to identify the limit.  Passing to the limit in the horizontal momentum equation gives
\[
        \partial_t v_h-\Delta v_h+(v_h\cdot\nabh)v_h+\nabh q=0.
\]
The vertical equation gives, in distributions,
\[
        \partial_3p^{(n)}
        =-\partial_tu^{(n)}_3+\Delta u^{(n)}_3-\nabla\cdot(u^{(n)}u^{(n)}_3),
\]
and the right-hand side tends to zero in distributions using \(u^{(n)}_3\to0\) in \(L^3\), the uniform energy bounds, and \eqref{eq:strong-L3}.  Hence \(\partial_3q=0\).

The suitability of the limit follows from the standard stability of the local energy inequality under this convergence.  Indeed, for every fixed nonnegative test function \(\phi\in C_c^\infty(Q_\theta)\), the terms containing \(|u^{(n)}|^2\) and \(|u^{(n)}|^2u^{(n)}\) pass to the limit by the strong \(L^3\)-convergence.  The pressure term also passes to the limit, since \(\widetilde p^{(n)}\rightharpoonup q\) in \(L^{3/2}\) and \(u^{(n)}\to v\) strongly in \(L^3\), so
\[
        \int \widetilde p^{(n)}u^{(n)}\cdot\nabla\phi\,\dxdt
        \longrightarrow
        \int qv\cdot\nabla\phi\,\dxdt .
\]
The subtracted time-dependent pressure averages do not contribute because \(u^{(n)}\) is divergence-free.  Lower semicontinuity gives the dissipative term.  Thus \((v,q)\in\calL_M(Q_\theta)\), after increasing \(K_0(M,\theta)\) if necessary.  This is incompatible with the lower bound on \(\calX^{\harm}_\theta\).

Define
\[
\omega_{M,\theta}(\delta)
=\sup\{\calX^{\harm}_\theta(u,p;M): (u,p)\text{ suitable in }Q_1,
\ \Phi(1)\le M,
\ C_3(1)\le\delta\}.
\]
The preceding contradiction argument proves \(\omega_{M,\theta}(\delta)\to0\) as \(\delta\downarrow0\).  The supremum is finite by a direct fixed-scale comparison: choose \((v,q)=(0,0)\in\calL_M(Q_\theta)\) and choose the harmonic corrector \(h=(p)_{B_\theta}(t)\), which is constant in space.  Then
\[
        \calX^{\harm}_\theta(u,p;M)
        \le C_\theta\{C(\theta)+D(\theta)\}
        \le C_\theta M,
\]
using fixed-scale nesting.  Replacing \(\omega_{M,\theta}\) by its monotone envelope gives a nondecreasing modulus.
\end{proof}

\begin{remark}
Strong convergence of the full pressure oscillation
\[
        p^{(n)}-(p^{(n)})_{B_\theta}(t)
        \longrightarrow q-(q)_{B_\theta}(t)
        \quad\text{in }L^{3/2}(Q_\theta)
\]
requires more than \(\Phi(1)\le M\) and strong \(L^3\)-convergence of velocities.  The Calderon--Zygmund part is strongly compact, while the harmonic part may fail to be strongly compact in time.  The quotient in \eqref{eq:harmonic-excess} records precisely the compact part that is stable under one-component convergence.
\end{remark}

\section{Unconditional finite-scale decay and regularity radius}
\label{sec:unconditional-decay}

\begin{theorem}[Finite-scale decay]
\label{thm:finite-decay-proof}
Let \(M\ge1\) and fix \(0<\theta<1/2\).  Let \(\omega_{M,\theta}\) be the modulus from Theorem~\ref{thm:qualitative-approx}.  There exist constants
\[
        K_C(M,\theta)\ge1,
        \qquad
        r_C(M,\theta)\in(0,\theta/2],
\]
such that every suitable weak solution \((u,p)\) in \(Q_1\) satisfying \(\Phi(1)\le M\) and \(C_3(1)\le\delta\) obeys
\begin{equation}
        \Psi(r)
        \le K_C(M,\theta)r+K_C(M,\theta)r^{-2}\omega_{M,\theta}(\delta)
        \label{eq:finite-decay}
\end{equation}
for every \(0<r<r_C(M,\theta)\).
\end{theorem}

\begin{proof}
Fix \(\eta>0\).  By Theorem~\ref{thm:qualitative-approx}, choose \((v,q)\in\calL_M(Q_\theta)\) and \(h\in\calH(Q_\theta)\) such that
\begin{equation}
\frac1{\theta^2}\int_{Q_\theta}|u-v|^3\dxdt
+
\frac1{\theta^2}\int_{Q_\theta}|p-q-h|^{3/2}\dxdt
\le \omega_{M,\theta}(\delta)+\eta.
\label{eq:chosen-comparison}
\end{equation}
Apply Proposition~\ref{prop:limiting-decay} after rescaling the limiting solution from \(Q_\theta\) to a unit cylinder.  Fixed-scale nesting gives constants \(K_A^*(M,\theta)\) and \(r_A^*(M,\theta)>0\) such that
\begin{equation}
        \Psi_v(r)
        \le K_A^*(M,\theta)r,
        \qquad
        0<r<r_A^*(M,\theta).
        \label{eq:limiting-rescaled-decay}
\end{equation}
We take \(r_C\le\min\{r_A^*,\theta/2,1\}\).

For \(0<r<r_C\), \eqref{eq:chosen-comparison} gives
\[
        \int_{Q_r}|u-v|^3\dxdt
        \le \int_{Q_\theta}|u-v|^3\dxdt
        \le \theta^2(\omega_{M,\theta}(\delta)+\eta),
\]
and hence, absorbing the fixed factor \(\theta^2\) into the constants depending on \(\theta\),
\[
\begin{aligned}
C_u(r)
&\le C C_v(r)+Cr^{-2}\int_{Q_r}|u-v|^3\dxdt\\
&\le CK_A^*(M,\theta)r+C_\theta r^{-2}(\omega_{M,\theta}(\delta)+\eta).
\end{aligned}
\]
For the pressure, set \(g=p-q-h\).  Then
\[
        p-(p)_{B_r}
        =(q-(q)_{B_r})+(h-(h)_{B_r})+(g-(g)_{B_r}).
\]
By \eqref{eq:limiting-rescaled-decay} and \eqref{eq:chosen-comparison},
\[
        D_q(r)\le K_A^*(M,\theta)r,
        \qquad
        D_g(r)\le C_\theta r^{-2}(\omega_{M,\theta}(\delta)+\eta).
\]
Indeed, the second estimate follows from
\[
        \int_{Q_r}|g|^{3/2}\dxdt
        \le \int_{Q_\theta}|g|^{3/2}\dxdt
        \le \theta^2(\omega_{M,\theta}(\delta)+\eta),
\]
together with the elementary inequality controlling oscillation around the mean.  Since \(\theta\) is fixed, this factor is absorbed into \(C_\theta\).
It remains to control \(h\).  Since \(h=p-q-g\), comparison of averages gives
\[
        D_h(\theta)\le C D_p(\theta)+C D_q(\theta)+C D_g(\theta).
\]
Here \(D_p(\theta)\le C(M,\theta)\), \(D_q(\theta)\le K_0(M,\theta)\), and \(D_g(\theta)\le C(\omega_{M,\theta}(\delta)+\eta)\).  Hence
\[
        D_h(\theta)
        \le C(M,\theta)+C(\omega_{M,\theta}(\delta)+\eta).
\]
Lemma~\ref{lem:harmonic-decay} gives, for \(0<r\le\theta/2\),
\[
        D_h(r)
        \le C\left(\frac r\theta\right)^{5/2}
        \{C(M,\theta)+\omega_{M,\theta}(\delta)+\eta\}.
\]
Since \(r<1\), this is bounded by
\[
        C(M,\theta)r+C(M,\theta)r^{-2}(\omega_{M,\theta}(\delta)+\eta).
\]
Combining the estimates for \(C_u,D_q,D_g\), and \(D_h\), then letting \(\eta\downarrow0\), proves \eqref{eq:finite-decay}.
\end{proof}

\begin{theorem}[Regularity radius]
\label{thm:regularity-radius}
Let \(M\ge1\).  There exist constants
\[
        \eps_D(M)>0,
        \qquad
        \rho_D(M)>0,
\]
such that, if \((u,p)\) is a suitable weak solution in \(Q_1\) satisfying
\[
        \Phi(1)\le M,
        \qquad
        C_3(1)\le\eps_D(M),
\]
then
\[
        r_{\reg}(0,0)\ge\rho_D(M).
\]
\end{theorem}

\begin{proof}
Fix \(\theta\in(0,1/2)\), for instance \(\theta=1/4\).  Let \(K_C=K_C(M,\theta)\) and \(r_C=r_C(M,\theta)\) be the constants from Theorem~\ref{thm:finite-decay-proof}.  Define
\[
        r_M=\min\left\{\frac{r_C}{2},\frac{\eps_{\CKN}}{2K_C}\right\}.
\]
Then \(K_Cr_M\le\eps_{\CKN}/2\).  Since \(\omega_{M,\theta}(s)\to0\) as \(s\downarrow0\), choose \(\eps_D(M)>0\) so small that
\[
        K_Cr_M^{-2}\omega_{M,\theta}(\eps_D(M))
        \le \frac12\eps_{\CKN}.
\]
If \(C_3(1)\le\eps_D(M)\), monotonicity of \(\omega_{M,\theta}\) and Theorem~\ref{thm:finite-decay-proof} give \(\Psi(r_M)\le\eps_{\CKN}\).  Theorem~\ref{thm:CKN} gives regularity in \(Q_{\kappa r_M}\).  Set \(\rho_D(M)=\kappa r_M\).
\end{proof}

Theorem~\ref{thm:finite-scale-main} follows from Theorem~\ref{thm:regularity-radius} with \(\eps_*(M)=\eps_D(M)\) and \(\rho_*(M)=\rho_D(M)\).  Theorem~\ref{thm:compactness-decay} is Theorem~\ref{thm:finite-decay-proof}.  Corollary~\ref{cor:concentration} follows by contradiction and scaling.

\section{A logarithmic refinement under a prepared comparison package}
\label{sec:log}

The preceding sections prove a finite-scale theorem with a qualitative compactness modulus.  This section records a semi-quantitative logarithmic refinement of the same mechanism.  The purpose is to isolate a precise stability estimate which, if proved, upgrades the abstract modulus in Theorem~\ref{thm:qualitative-approx} to a logarithmic rate.

Throughout this section write
\[
        \delta=C_3(1),
\]
and fix \(M\ge1\) and \(0<\theta<1/2\).  Constants may depend on \(M\) and \(\theta\), but not on \(\delta\) or the smoothing parameter \(\ell\).  We use the parabolic energy norm
\[
        \norm{W}{\calZ(I\times B)}^2
        =\norm{W}{L^\infty_tL^2_x(I\times B)}^2
        +\norm{\nabla W}{L^2_{t,x}(I\times B)}^2.
\]
The dual residual norm \(\calY\) is any norm for which
\[
        |\langle F,W\rangle|
        \le \frac12\norm{\nabla W}{L^2}^2+C\norm{F}{\calY}^2
\]
for the localized test functions used below.  An \(L^2_tH^{-1}_x\)-type norm is a useful model.

\begin{proposition}[Model two-shadow stability]
\label{prop:two-shadow}
Let \(I=(s,T)\), and let \(B\Subset\R^3\) be a ball.  This is a model estimate in an ideal localized setting.  Assume that \(U,V,V^\ell\) are horizontal vector fields on \(I\times B\), horizontally divergence-free, and that either no boundary terms arise, for instance in the full-space, periodic, or compactly supported setting, or that the localization has already been accompanied by a horizontal Helmholtz/Bogovskii-type correction so that the test fields below remain horizontally divergence-free and the pressure pairings vanish.  Under these corrected local equations, assume
\begin{align}
\partial_tU-\Delta U+(U\cdot\nabh)U+\nabh P&=F_\delta,\label{eq:U-shadow}\\
\partial_tV-\Delta V+(V\cdot\nabh)V+\nabh Q&=0,\label{eq:rough-shadow}\\
\partial_tV^\ell-\Delta V^\ell+(V^\ell\cdot\nabh)V^\ell+\nabh Q^\ell&=0.\label{eq:smooth-shadow}
\end{align}
Assume also that
\[
        V(s)=U(s),
        \qquad
        V^\ell(s)=J_\ell U(s),
\]
where \(J_\ell\) is a smoothing operator, and that
\begin{equation}
        \int_s^T\norm{\nabh V(t)}{L^\infty(B)}\dt\le\Lambda.
        \label{eq:rough-bound}
\end{equation}
Finally assume
\begin{equation}
        \norm{F_\delta}{\calY(I\times B)}\le A_\delta,
        \qquad
        \norm{U(s)-J_\ell U(s)}{L^2(B)}\le A_\ell.
        \label{eq:two-shadow-data}
\end{equation}
Then, on every smaller cylinder \(I'\times B'\Subset I\times B\),
\begin{equation}
        \norm{U-V^\ell}{\calZ(I'\times B')}
        \le C_{B',B,I',I}e^{C\Lambda}(A_\ell+A_\delta).
        \label{eq:two-shadow-estimate}
\end{equation}
Thus, in this model or corrected local setting, the smoothing error \(A_\ell\) is multiplied only by the delayed norm of the rough shadow \(V\), while the high norm of \(V^\ell\) is avoided.
\end{proposition}

\begin{proof}
Under the ideal or already-corrected localization assumptions in the statement, the energy test by the difference fields produces no uncompensated boundary, divergence, or pressure terms.  Without such a correction, the following computation should be read only as a formal model; the closed localized estimate is precisely part of the prepared comparison input in Assumption~\ref{ass:prepared-comparison}.

Set \(W=U-V\).  Subtracting \eqref{eq:rough-shadow} from \eqref{eq:U-shadow} gives
\[
        \partial_tW-\Delta W+(U\cdot\nabh)W+(W\cdot\nabh)V+\nabh(P-Q)=F_\delta.
\]
Testing by \(W\), the transport term with velocity \(U\) vanishes because \(\nabh\cdot U=0\), and the pressure term vanishes because \(W\) is horizontally divergence-free.  Therefore
\[
        \frac12\frac{d}{dt}\norm{W}{L^2}^2+
\norm{\nabla W}{L^2}^2
        \le \norm{\nabh V}{L^\infty}\norm{W}{L^2}^2+\langle F_\delta,W\rangle.
\]
By the definition of \(\calY\),
\[
        \langle F_\delta,W\rangle
        \le \frac12\norm{\nabla W}{L^2}^2+C\norm{F_\delta}{\calY}^2.
\]
Since \(W(s)=0\), Gronwall's inequality and \eqref{eq:rough-bound} yield
\begin{equation}
        \norm{U-V}{\calZ(I'\times B')}
        \le Ce^{C\Lambda}A_\delta.
        \label{eq:U-V}
\end{equation}
Next set \(Z=V-V^\ell\).  Subtracting \eqref{eq:smooth-shadow} from \eqref{eq:rough-shadow}, and writing
\[
        (V\cdot\nabh)V-(V^\ell\cdot\nabh)V^\ell
        =(V^\ell\cdot\nabh)Z+(Z\cdot\nabh)V,
\]
one obtains
\[
        \partial_tZ-\Delta Z+(V^\ell\cdot\nabh)Z+(Z\cdot\nabh)V+\nabh(Q-Q^\ell)=0.
\]
Testing by \(Z\), the transport term with velocity \(V^\ell\) vanishes because \(\nabh\cdot V^\ell=0\), and the pressure term vanishes.  Hence
\[
        \frac12\frac{d}{dt}\norm{Z}{L^2}^2+
\norm{\nabla Z}{L^2}^2
        \le \norm{\nabh V}{L^\infty}\norm{Z}{L^2}^2.
\]
Since \(Z(s)=U(s)-J_\ell U(s)\), Gronwall gives
\begin{equation}
        \norm{V-V^\ell}{\calZ(I'\times B')}
        \le Ce^{C\Lambda}A_\ell.
        \label{eq:V-Vell}
\end{equation}
Combining \eqref{eq:U-V} and \eqref{eq:V-Vell} proves \eqref{eq:two-shadow-estimate}.
\end{proof}

\begin{remark}
The identity
\[
(V\cdot\nabh)V-(V^\ell\cdot\nabh)V^\ell
=(V^\ell\cdot\nabh)(V-V^\ell)+((V-V^\ell)\cdot\nabh)V
\]
is the key point.  In the ideal or divergence-corrected setting, the term containing \(V^\ell\) cancels in the energy estimate by horizontal incompressibility.  The coefficient that remains is \(\nabh V\), with the high coefficient \(\nabh V^\ell\) absent.  Thus the smoothing error is propagated by the rough shadow.  In a bounded cylinder, the cutoff, pressure, and horizontal-divergence corrections must be supplied by the prepared comparison package rather than by this model calculation alone.
\end{remark}

\begin{assumption}[Prepared comparison estimate]
\label{ass:prepared-comparison}
For every \(M\ge1\) and \(0<\theta<1/2\), there exist constants
\[
        C_{M,\theta}\ge1,
        \qquad
        \ell_0=\ell_0(M,\theta)\in(0,1),
        \qquad
        a,b,N>0,
\]
with the following property.  If \((u,p)\) is suitable in \(Q_1\), \(\Phi(1)\le M\), and \(\delta=C_3(1)\), then for every \(0<\ell<\ell_0\) there exist a solution \((v^\ell,q^\ell)\in\calL_M(Q_{\theta/4})\) of the limiting system and a harmonic pressure corrector \(h^\ell\in\calH(Q_{\theta/4})\) such that
\begin{equation}
\calE^{\harm}_{\theta/4}((u,p),(v^\ell,q^\ell);h^\ell)
\le C_{M,\theta}\left(\ell^a+\ell^{-N}\delta^b+\exp(C_{M,\theta}\ell^{-N})\delta^b\right).
\label{eq:prepared-comparison}
\end{equation}
\end{assumption}

\begin{remark}
The term \(\ell^a\) represents preparation and smoothing error.  The terms containing \(\delta^b\) represent the vertical transport residual, horizontal divergence correction, and pressure components generated by factors containing \(u_3\).  The exponential term allows a weak--strong stability estimate in which the relevant smooth comparison norm behaves like \(\ell^{-N}\).  The model estimate in Proposition~\ref{prop:two-shadow} explains why the smoothing error itself is kept outside that exponential factor; the actual localized cutoff and pressure corrections are part of the present assumption.
\end{remark}

\begin{theorem}[Logarithmic harmonic-pressure approximation]
\label{thm:log-approx-proof}
Suppose Assumption~\ref{ass:prepared-comparison} holds.  Then there exist constants \(C'_{M,\theta}\ge1\), \(\sigma>0\), and \(\delta_{M,\theta}\in(0,1)\) such that every suitable weak solution \((u,p)\) in \(Q_1\) satisfying \(\Phi(1)\le M\) and \(\delta=C_3(1)\le\delta_{M,\theta}\) obeys
\begin{equation}
        \calX^{\harm}_{\theta/4}(u,p;M)
        \le C'_{M,\theta}|\log\delta|^{-\sigma}.
        \label{eq:log-approx}
\end{equation}
One may take \(\sigma=a/N\), after decreasing \(\delta_{M,\theta}\) if necessary.
\end{theorem}

\begin{proof}
By Assumption~\ref{ass:prepared-comparison},
\[
        \calX^{\harm}_{\theta/4}(u,p;M)
        \le C_{M,\theta}\left(\ell^a+\ell^{-N}\delta^b+\exp(C_{M,\theta}\ell^{-N})\delta^b\right)
\]
for every \(0<\ell<\ell_0\).  Set \(L=|\log\delta|\).  For \(\delta\) sufficiently small, choose
\[
        \ell=\left(\frac{2C_{M,\theta}}{bL}\right)^{1/N}.
\]
Then \(0<\ell<\ell_0\), and
\[
        C_{M,\theta}\ell^{-N}=\frac b2L.
\]
Consequently,
\[
        \exp(C_{M,\theta}\ell^{-N})\delta^b
        =\exp\left(\frac b2L\right)\exp(-bL)
        =\exp\left(-\frac b2L\right)
        =\delta^{b/2}.
\]
Moreover
\[
        \ell^a
        =\left(\frac{2C_{M,\theta}}{bL}\right)^{a/N}
        \le C_{M,\theta}L^{-a/N},
\]
and
\[
        \ell^{-N}\delta^b
        =\frac{bL}{2C_{M,\theta}}e^{-bL}
        \le C L^{-a/N}
\]
for all sufficiently large \(L\).  Also \(\delta^{b/2}\le C L^{-a/N}\).  Hence \eqref{eq:log-approx} holds with \(\sigma=a/N\).
\end{proof}

\begin{corollary}[Logarithmic finite-scale decay and radius]
\label{cor:log-decay}
Suppose Assumption~\ref{ass:prepared-comparison} holds.  For every \(M\ge1\) and \(0<\theta<1/2\), there exist constants \(C_{M,\theta}\ge1\), \(\sigma>0\), and \(\delta_{M,\theta}\in(0,1)\) such that if \(\Phi(1)\le M\) and \(\delta=C_3(1)\le\delta_{M,\theta}\), then
\begin{equation}
        \Psi(r)
        \le C_{M,\theta}r+C_{M,\theta}r^{-2}|\log\delta|^{-\sigma}
        \label{eq:log-decay}
\end{equation}
for all sufficiently small \(r\).  In particular, after decreasing \(\delta_{M,\theta}\),
\begin{equation}
        r_{\reg}(0,0)
        \ge c_{M,\theta}|\log\delta|^{-\sigma/3}.
        \label{eq:log-radius}
\end{equation}
\end{corollary}

\begin{proof}
The decay estimate \eqref{eq:log-decay} follows by repeating the proof of Theorem~\ref{thm:finite-decay-proof}, replacing \(\omega_{M,\theta}(\delta)\) by the logarithmic bound \eqref{eq:log-approx} at scale \(\theta/4\).  Let \(L=|\log\delta|\), and choose \(r=L^{-\sigma/3}\).  Then
\[
        r^{-2}L^{-\sigma}=L^{2\sigma/3}L^{-\sigma}=L^{-\sigma/3},
        \qquad
        r=L^{-\sigma/3}.
\]
Thus \eqref{eq:log-decay} gives \(\Psi(r)\le C_{M,\theta}L^{-\sigma/3}\).  For \(L\) sufficiently large, this is below \(\eps_{\CKN}\), and Theorem~\ref{thm:CKN} gives \eqref{eq:log-radius}.
\end{proof}

Theorem~\ref{thm:log-main} is Theorem~\ref{thm:log-approx-proof} and Corollary~\ref{cor:log-decay}.

\section{Relaxed shadowing and conditional power-type approximation}
\label{sec:relaxed}

The compactness and logarithmic layers compare the solution with the strict limiting system \eqref{eq:limiting-system}.  The relaxed-shadowing layer enlarges the comparison class.  The goal is to avoid a nonlinear projection onto the strict compatibility condition \(\partial_3q=0\).

\subsection{Weighted vertical pressure compatibility}

Let \(Q=\omega_h\times I_3\times I_t\Subset Q_{3/4}\), where \(\omega_h\subset\R^2\) and \(I_3=(a,b)\).  Choose \(\zeta\in C_c^\infty(I_3)\) with \(\int_{I_3}\zeta=1\).  For a scalar \(f\), define
\[
        \langle f\rangle_\zeta(x_h,t)
        =\int_{I_3}\zeta(y)f(x_h,y,t)\,dy,
        \qquad
        \Pi_\zeta f=f-\langle f\rangle_\zeta.
\]
For \(\varphi_h\in C_c^\infty(Q;\R^2)\), set \(f_\varphi=\nabh\cdot\varphi_h\), and define
\[
        g_\varphi(x_h,x_3,t)
        =f_\varphi(x_h,x_3,t)
        -\zeta(x_3)\int_{I_3}f_\varphi(x_h,y,t)\,dy.
\]
Then \(\int_{I_3}g_\varphi\,dx_3=0\).  The vertical antiderivative
\[
        B_\varphi(x_h,x_3,t)=\int_a^{x_3}g_\varphi(x_h,s,t)\,ds
\]
is compactly supported in the vertical variable and satisfies \(\partial_3B_\varphi=g_\varphi\).  Indeed, \(g_\varphi\) has zero \(x_3\)-average and is supported away from the endpoints of \(I_3\), because both \(\varphi_h\) and \(\zeta\) are compactly supported in the vertical interval; hence \(B_\varphi\) vanishes near \(a\) and near \(b\).  Define
\[
\norm{\varphi_h}{Y_\zeta(Q)}
=\norm{(\partial_t+\Delta)B_\varphi}{L^{3/2}(Q)}
+\norm{\nabh B_\varphi}{L^3(Q)}
+\norm{\partial_3B_\varphi}{L^3(Q)}.
\]

Let \(\chi\in C_c^\infty(B_1)\) satisfy \(\chi\equiv1\) on \(B_{3/4}\).  Write
\[
        p^{hh}=R_aR_b(\chi u_au_b),
        \qquad
        p^{\rem}=2R_aR_3(\chi u_au_3)+R_3R_3(\chi u_3^2),
        \qquad a,b\in\{1,2\}.
\]
After subtracting a function of time,
\[
        p=p^{hh}+p^{\rem}+h
        \quad\text{in }Q_{3/4},
\]
where \(h(\cdot,t)\) is harmonic in space.  Set
\[
        P=p^{hh}+h=p-p^{\rem}.
\]

\begin{lemma}[Vertical pressure remainder]
\label{lem:vertical-pressure-remainder}
If \(\Phi(1)\le M\) and \(C_3(1)=\delta\le1\), then
\[
        \norm{p^{\rem}}{L^{3/2}(Q_{3/4})}
        \le C(M)\delta^{1/3}.
\]
The same estimate holds on every smaller cylinder \(Q_\rho\Subset Q_{3/4}\), with a constant depending on \(\rho\).
\end{lemma}

\begin{proof}
The boundedness of Riesz transforms on \(L^{3/2}\) gives
\[
        \norm{p^{\rem}}{L^{3/2}}
        \le C\norm{u_hu_3}{L^{3/2}}+C\norm{u_3^2}{L^{3/2}}.
\]
Holder's inequality yields
\[
        \norm{u_hu_3}{L^{3/2}(Q_1)}
        \le \norm{u_h}{L^3(Q_1)}\norm{u_3}{L^3(Q_1)}
        \le C(M)\delta^{1/3},
\]
and \(\norm{u_3^2}{L^{3/2}}=\norm{u_3}{L^3}^2=\delta^{2/3}\le\delta^{1/3}\).
\end{proof}

\begin{lemma}[Weighted vertical pressure compatibility]
\label{lem:weighted-compatibility}
For every product cylinder \(Q\Subset Q_{3/4}\) and every \(\zeta\in C_c^\infty(I_3)\) with \(\int_{I_3}\zeta=1\),
\[
        \norm{\nabh\Pi_\zeta P}{Y'_\zeta(Q)}
        \le C(M,Q,\zeta)\delta^{1/3}.
\]
Equivalently,
\[
        \norm{\nabh(P-\langle P\rangle_\zeta)}{Y'_\zeta(Q)}
        \le C(M,Q,\zeta)C_3(1)^{1/3}.
\]
\end{lemma}

\begin{proof}
For \(\varphi_h\in C_c^\infty(Q;\R^2)\), integration by parts gives
\begin{equation}
        \langle\nabh\Pi_\zeta P,\varphi_h\rangle
        =\langle\partial_3P,B_\varphi\rangle.
        \label{eq:vertical-antiderivative-identity}
\end{equation}
Indeed,
\[
\langle\nabh\Pi_\zeta P,\varphi_h\rangle
= -\int_Q \Pi_\zeta P f_\varphi\dxdt
= -\int_Q P g_\varphi\dxdt
= -\int_Q P\partial_3B_\varphi\dxdt.
\]
Since \(B_\varphi\) is compactly supported, this equals \(\langle\partial_3P,B_\varphi\rangle\).

Because \(P=p-p^{\rem}\), the vertical momentum equation gives
\[
\partial_3P
= -\partial_tu_3+\Delta u_3-\nabh\cdot(u_hu_3)-\partial_3(u_3^2)-\partial_3p^{\rem}.
\]
Pairing with \(B_\varphi\) and moving derivatives to the test function gives
\[
\begin{aligned}
|\langle\partial_3P,B_\varphi\rangle|
&\le \norm{u_3}{L^3}\norm{(\partial_t+\Delta)B_\varphi}{L^{3/2}}
+\norm{u_hu_3}{L^{3/2}}\norm{\nabh B_\varphi}{L^3}\\
&\quad +\norm{u_3^2}{L^{3/2}}\norm{\partial_3B_\varphi}{L^3}
+\norm{p^{\rem}}{L^{3/2}}\norm{\partial_3B_\varphi}{L^3}.
\end{aligned}
\]
Using \(\Phi(1)\le M\), \(C_3(1)=\delta\), and Lemma~\ref{lem:vertical-pressure-remainder}, we obtain
\[
        |\langle\nabh\Pi_\zeta P,\varphi_h\rangle|
        \le C(M,Q,\zeta)\delta^{1/3}\norm{\varphi_h}{Y_\zeta(Q)}.
\]
Taking the supremum proves the lemma.
\end{proof}

\subsection{Relaxed no-stretching comparison}

The relaxed comparison system is
\begin{equation}
\begin{cases}
\partial_t v_h-\Delta v_h+(v_h\cdot\nabh)v_h+\nabh\pi=0,\\
\nabh\cdot v_h=0.
\end{cases}
\label{eq:RNS}
\end{equation}
Set \(V=(v_h,0)\).  Then \(\nabla\cdot V=0\), and
\begin{equation}
        \partial_tV-\Delta V+(V\cdot\nabla)V+\nabla\pi
        =(0,0,\partial_3\pi).
        \label{eq:RNS-3D}
\end{equation}
The pressure is allowed to satisfy \(\partial_3\pi\ne0\).  The vertical vorticity \(\omega=\partial_1v_2-\partial_2v_1\) satisfies
\[
        \partial_t\omega-\Delta\omega+v_h\cdot\nabh\omega=0,
\]
and the three-dimensional vortex-stretching term is absent.

\begin{assumption}[Buffered relaxed strong bound]
\label{ass:relaxed-strong}
For every \(M\ge1\) and every sufficiently small fixed \(\theta>0\), the solution \((v_h,\pi)\) of \eqref{eq:RNS} generated from the good-time horizontally divergence-free datum constructed below exists on a forward buffered cylinder containing the target cylinder, and satisfies
\begin{equation}
        \norm{V}{C^1}+\norm{\pi}{C^1}
        \le K_B(M,\theta)
        \label{eq:buffered-bound}
\end{equation}
on that buffered cylinder.
\end{assumption}

\begin{remark}
This assumption is suggested by the scalar no-stretching vorticity equation and the parabolic buffer between the good time and the target cylinder.  A fully unconditional power-rate theorem requires this local buffered regularity estimate, either proved directly or supplied by an applicable cited theorem with matching local hypotheses.
\end{remark}

Fix \(0<\theta<1/16\), and set
\[
        I_-:=(-16\theta^2,-9\theta^2).
\]
Averaging the local energy and vertical-component bounds gives a good time \(s\in I_-\) such that
\begin{equation}
        \int_{B_{4\theta}}|u(x,s)|^2\dx
        +\theta^2\int_{B_{4\theta}}|\nabla u(x,s)|^2\dx
        \le C(M,\theta),
        \label{eq:good-time-energy}
\end{equation}
and
\begin{equation}
        \int_{B_{4\theta}}|u_3(x,s)|^3\dx
        \le C(\theta)\delta.
        \label{eq:good-time-u3}
\end{equation}
Using a finite family of product charts and horizontal Helmholtz projections on each \(x_3\)-slice, one obtains a patched horizontal datum \(a_s\) with \(\nabh\cdot a_s=0\) locally and
\begin{equation}
        \norm{u_h(s)-a_s}{H^{-m}(B_{3\theta})}
        +\norm{u_3(s)}{H^{-m}(B_{3\theta})}
        \le C(M,\theta)\delta^{1/3}
        \label{eq:projection-error}
\end{equation}
for \(m\) sufficiently large.  Indeed, in a product chart the horizontal projection error is generated by
\[
        u_h-\Ph u_h
        =\nabh\Delta_{h,D}^{-1}\nabh\cdot u_h
        =-\nabh\Delta_{h,D}^{-1}\partial_3u_3,
\]
and the derivative is transferred to a smooth test function in the \(H^{-m}\)-pairing.

Let \((v_h,\pi)\) solve \eqref{eq:RNS} with \(v_h(s)=a_s\), and set \(V=(v_h,0)\).

\begin{assumption}[Localized relaxed weak--strong stability]
\label{ass:relaxed-stability}
Assume \(\Phi(1)\le M\), \(C_3(1)=\delta\le1\), and Assumption~\ref{ass:relaxed-strong}.  For the relaxed comparison field \(V\) constructed from the good-time projected datum, there exist constants \(C_S(M,\theta)\ge1\) and \(\gamma>0\) such that
\begin{equation}
        \norm{u-V}{L^3(Q_{2\theta})}
        \le C_S(M,\theta)\delta^\gamma.
        \label{eq:relaxed-stability}
\end{equation}
\end{assumption}

\begin{remark}
Formally, this is suggested by a localized relative-energy inequality for \(W=u-V\).  The residual in \eqref{eq:RNS-3D} contributes
\[
        \int \partial_3\pi\,u_3,
\]
which is small because \(\partial_3\pi\) is bounded by Assumption~\ref{ass:relaxed-strong} and \(u_3\) is small in \(L^3\).  In a bounded local cylinder, cutoff and pressure terms must be treated carefully.  Assumption~\ref{ass:relaxed-stability} records the closed local estimate needed for the conditional theorem.
\end{remark}

\subsection{Pressure reconstruction and power-type decay}

The pressure \(\pi\) in \eqref{eq:RNS} differs from the pressure entering \(D(r)\).  Define the local Navier--Stokes-compatible pressure generated by \(V\):
\[
        Q_V=R_iR_j(\chi V_iV_j)=R_aR_b(\chi v_av_b),
        \qquad a,b\in\{1,2\},
\]
where \(\chi\) is equal to one on the target interior cylinder.

\begin{proposition}[Pressure reconstruction]
\label{prop:pressure-reconstruction}
Under Assumptions~\ref{ass:relaxed-strong} and \ref{ass:relaxed-stability}, there exists a pressure corrector \(h\), harmonic in space on \(Q_\theta\), such that
\begin{equation}
        \norm{p-Q_V-h}{L^{3/2}(Q_\theta)}
        \le C(M,\theta)\norm{u-V}{L^3(Q_{2\theta})}.
        \label{eq:pressure-reconstruction}
\end{equation}
Consequently,
\[
        \norm{p-Q_V-h}{L^{3/2}(Q_\theta)}
        \le C(M,\theta)\delta^\gamma,
\]
with \(\gamma\) as in \eqref{eq:relaxed-stability}.
\end{proposition}

\begin{proof}
Use the local pressure decomposition
\[
        p=R_iR_j(\chi u_iu_j)+h,
\]
where \(h\) is harmonic in space on the interior cylinder.  Then
\[
        p-Q_V-h
        =R_iR_j(\chi(u_iu_j-V_iV_j)).
\]
Calderon--Zygmund estimates give
\[
        \norm{p-Q_V-h}{L^{3/2}(Q_\theta)}
        \le C\norm{u\otimes u-V\otimes V}{L^{3/2}(Q_{2\theta})}.
\]
Since
\[
        u\otimes u-V\otimes V=(u-V)\otimes u+V\otimes(u-V),
\]
the right-hand side is bounded by \(C(M,\theta)\norm{u-V}{L^3(Q_{2\theta})}\), using the local \(L^3\)-bound for \(u\) and the buffered strong bound for \(V\).
\end{proof}

Define the relaxed harmonic excess
\[
\calE^{\rel,\harm}_\theta(u,p;V,Q_V,h)
=\theta^{-2}\int_{Q_\theta}|u-V|^3\dxdt
+\theta^{-2}\int_{Q_\theta}|p-Q_V-h|^{3/2}\dxdt.
\]

\begin{theorem}[Hybrid relaxed harmonic approximation]
\label{thm:relaxed-harmonic-approx}
Assume \(\Phi(1)\le M\), \(C_3(1)=\delta\le1\), and Assumptions~\ref{ass:relaxed-strong} and \ref{ass:relaxed-stability}.  Then there exist a smooth relaxed no-stretching comparison pair \((V,Q_V)\) and a spatially harmonic pressure corrector \(h\in\calH(Q_\theta)\) such that
\begin{equation}
        \calE^{\rel,\harm}_\theta(u,p;V,Q_V,h)
        \le C(M,\theta)\delta^\Gamma
        \label{eq:relaxed-harmonic-approx}
\end{equation}
for some \(\Gamma>0\) depending only on the exponent in Assumption~\ref{ass:relaxed-stability}.
\end{theorem}

\begin{proof}
By Assumption~\ref{ass:relaxed-stability}, \(\norm{u-V}{L^3(Q_\theta)}\le C(M,\theta)\delta^\gamma\).  Hence
\[
        \theta^{-2}\int_{Q_\theta}|u-V|^3\dxdt
        \le C(M,\theta)\delta^{3\gamma}.
\]
By Proposition~\ref{prop:pressure-reconstruction},
\[
        \theta^{-2}\int_{Q_\theta}|p-Q_V-h|^{3/2}\dxdt
        \le C(M,\theta)\delta^{3\gamma/2}.
\]
Set \(\Gamma=\min\{3\gamma,3\gamma/2\}\), decreasing it if necessary.
\end{proof}

The comparison pair is smooth in the buffered target cylinder.  Thus there exist \(\alpha>0\) and \(C(M,\theta)\) such that
\begin{equation}
        C_V(r)+D_{Q_V}(r)
        \le C(M,\theta)r^\alpha,
        \qquad 0<r<r_H(M,\theta).
        \label{eq:smooth-comparison-decay}
\end{equation}
For a \(C^1\)-comparison one may take any fixed \(\alpha\le1\); with more regularity one obtains better exponents.  The argument below only uses positivity of this exponent.

\begin{theorem}[Conditional hybrid power-type finite-scale decay]
\label{thm:power-decay-proof}
Assume Assumptions~\ref{ass:relaxed-strong} and \ref{ass:relaxed-stability}.  There exist constants \(C_H(M,\theta)\ge1\), \(\alpha>0\), \(\Gamma>0\), and \(r_H(M,\theta)>0\) such that every suitable weak solution in \(Q_1\) with \(\Phi(1)\le M\) and \(C_3(1)=\delta\le1\) obeys
\begin{equation}
        \Psi(r)
        \le C_H(M,\theta)r^\alpha+C_H(M,\theta)r^{-2}\delta^\Gamma,
        \qquad 0<r<r_H(M,\theta).
        \label{eq:power-decay}
\end{equation}
\end{theorem}

\begin{proof}
Let \((V,Q_V,h)\) be the comparison triple from Theorem~\ref{thm:relaxed-harmonic-approx}.  For the velocity part,
\[
        |u|^3\le C|V|^3+C|u-V|^3,
\]
so
\[
        C_u(r)
        \le C C_V(r)+Cr^{-2}\int_{Q_r}|u-V|^3\dxdt.
\]
Using \eqref{eq:smooth-comparison-decay} and \eqref{eq:relaxed-harmonic-approx},
\[
        C_u(r)
        \le C(M,\theta)r^\alpha+C(M,\theta)r^{-2}\delta^\Gamma.
\]
For the pressure, write
\[
        p=Q_V+h+g,
        \qquad
        g=p-Q_V-h.
\]
The comparison pressure decays by \eqref{eq:smooth-comparison-decay}.  The harmonic part is controlled by Lemma~\ref{lem:harmonic-decay}; its fixed-scale oscillation is bounded by \(C(M,\theta)\) from the decomposition and \eqref{eq:relaxed-harmonic-approx}.  The error \(g\) contributes \(Cr^{-2}\delta^\Gamma\).  Thus
\[
        D_p(r)
        \le C(M,\theta)r^\alpha+C(M,\theta)r^{-2}\delta^\Gamma.
\]
Adding the estimates for \(C_u(r)\) and \(D_p(r)\) gives \eqref{eq:power-decay}.
\end{proof}

\begin{corollary}[Conditional power-type CKN scales]
\label{cor:power-radius}
Under Assumptions~\ref{ass:relaxed-strong} and \ref{ass:relaxed-stability}, the decay estimate \eqref{eq:power-decay} yields an explicit admissible CKN scale
\[
        r_\delta\simeq \delta^{\Gamma/(\alpha+2)}.
\]
Consequently, there are constants \(c_H(M,\theta)>0\) and \(\delta_H(M,\theta)>0\) such that, if \(0<C_3(1)=\delta\le\delta_H(M,\theta)\), then
\[
        r_{\reg}(0,0)
        \ge c_H(M,\theta)\delta^{\Gamma/(\alpha+2)}.
\]
This traceable scale is not intended to be optimal for very small \(\delta\).  The same decay estimate also gives the stronger fixed-radius statement: after possibly decreasing \(\delta_H(M,\theta)\), there exists \(\rho_H(M,\theta)>0\) such that
\[
        r_{\reg}(0,0)
        \ge \rho_H(M,\theta)
\]
for all \(0<\delta\le\delta_H(M,\theta)\).
\end{corollary}

\begin{proof}
Fix a constant \(0<c_0\le 1\), for instance \(c_0=1\), and set
\[
        r_\delta=c_0\delta^{\Gamma/(\alpha+2)}.
\]
Choose \(\delta_H\in(0,1]\) sufficiently small so that
\[
        c_0\delta_H^{\Gamma/(\alpha+2)}<r_H(M,\theta)
\]
and
\[
        C_H(M,\theta)\bigl(c_0^\alpha+c_0^{-2}\bigr)
        \delta_H^{\alpha\Gamma/(\alpha+2)}
        \le \eps_{\CKN}.
\]
Then, for every \(0<\delta\le\delta_H\), one has \(0<r_\delta<r_H(M,\theta)\).  Applying \eqref{eq:power-decay} at the scale \(r_\delta\) gives
\[
\begin{aligned}
        \Psi(r_\delta)
        &\le C_H r_\delta^\alpha+C_Hr_\delta^{-2}\delta^\Gamma  \\
        &= C_Hc_0^\alpha\delta^{\alpha\Gamma/(\alpha+2)}
        +C_Hc_0^{-2}\delta^{-2\Gamma/(\alpha+2)}\delta^\Gamma  \\
        &= C_H\bigl(c_0^\alpha+c_0^{-2}\bigr)
        \delta^{\alpha\Gamma/(\alpha+2)}
        \le \eps_{\CKN}.
\end{aligned}
\]
Theorem~\ref{thm:CKN} gives
\[
        r_{\reg}(0,0)\ge \kappa r_\delta
        =\kappa c_0\delta^{\Gamma/(\alpha+2)}.
\]
Thus the first conclusion holds with \(c_H(M,\theta)=\kappa c_0\).

For the fixed-radius statement, choose
\[
        r_M=\min\left\{\frac{r_H(M,\theta)}2,
        \left(\frac{\eps_{\CKN}}{2C_H(M,\theta)}\right)^{1/\alpha}\right\}.
\]
Then \(C_Hr_M^\alpha\le\eps_{\CKN}/2\).  Decrease \(\delta_H\), if necessary, so that
\[
        C_Hr_M^{-2}\delta_H^\Gamma\le\eps_{\CKN}/2.
\]
For every \(0<\delta\le\delta_H\), \eqref{eq:power-decay} gives \(\Psi(r_M)\le\eps_{\CKN}\).  Hence Theorem~\ref{thm:CKN} gives
\[
        r_{\reg}(0,0)\ge \rho_H(M,\theta),
        \qquad \rho_H(M,\theta)=\kappa r_M.
\]
\end{proof}

Theorem~\ref{thm:power-main} follows from Theorem~\ref{thm:power-decay-proof} and Corollary~\ref{cor:power-radius}.

\section{Remarks and limitations}
\label{sec:remarks}

\begin{remark}[Finite-scale and qualitative character]
The unconditional theorem is finite-scale because it produces a positive radius \(\rho_*(M)\) at which regularity is guaranteed.  The constants \(\rho_*(M)\) and \(\eps_*(M)\) are obtained through the compactness modulus \(\omega_{M,\theta}\), so the statement remains qualitative.
\end{remark}

\begin{remark}[Role of harmonic pressure]
The proof keeps estimates of \(\norm{\nabla q}{L^\infty}\) for the complete limiting pressure outside the argument.  The non-harmonic pressure is controlled by local Calderon--Zygmund estimates in finite \(L^s\)-spaces.  The harmonic part is kept in the excess and later controlled by harmonic oscillation decay.
\end{remark}

\begin{remark}[What the logarithmic layer proves]
The logarithmic layer records Assumption~\ref{ass:prepared-comparison} as an input and shows that a specific prepared comparison estimate, with an exponential-in-\(\ell^{-1}\) stability loss but with the smoothing error outside that exponential, is enough to replace the compactness modulus by a logarithmic one.  This isolates the next quantitative problem.
\end{remark}

\begin{remark}[What the relaxed layer proves]
The relaxed layer uses a smooth no-stretching comparison pair modulo harmonic pressure in place of an exact strict limiting-system solution satisfying \(\partial_3q=0\).  This relaxed comparison is sufficient for \(\Psi=C+D\)-decay because the CKN criterion requires smallness of \(C+D\) at one scale.
\end{remark}

\begin{remark}[Remaining quantitative inputs]
The remaining inputs for a power-rate theorem are Assumptions~\ref{ass:relaxed-strong} and \ref{ass:relaxed-stability}.  The first asks for a buffered strong bound for the relaxed no-stretching comparison flow.  The second asks for a localized relaxed weak--strong stability estimate controlling cutoff and pressure terms.  Once these inputs are available, the pressure reconstruction argument gives the power-type decay estimate \eqref{eq:power-decay}.
\end{remark}

\begin{remark}[Possible refinements]
A stronger quantitative theorem may follow from one of two approaches.  The first is to prove the buffered relaxed strong bound directly from the scalar no-stretching vorticity equation and close the localized relaxed weak--strong stability estimate with cutoffs and pressure terms.  The second is to establish a Reynolds-stress harmonic approximation theorem in which the smoothing commutator contributes as a static subscale energy with no Gronwall amplification.  Both approaches aim to replace the abstract compactness modulus by a traceable logarithmic or power rate.
\end{remark}

\backmatter


\begin{thebibliography}{99}
	
	\bibitem{Leray1934}
	J. Leray,
	Sur le mouvement d'un liquide visqueux emplissant l'espace,
	\emph{Acta Mathematica} \textbf{63} (1934), 193--248.
	\newblock DOI: \url{https://doi.org/10.1007/BF02547354}.
	
	\bibitem{Hopf1951}
	E. Hopf,
	\"Uber die Anfangswertaufgabe f\"ur die hydrodynamischen Grundgleichungen,
	\emph{Mathematische Nachrichten} \textbf{4} (1951), 213--231.
	\newblock DOI: \url{https://doi.org/10.1002/mana.3210040121}.
	
	\bibitem{Prodi1959}
	G. Prodi,
	Un teorema di unicit\`a per le equazioni di Navier--Stokes,
	\emph{Annali di Matematica Pura ed Applicata} \textbf{48} (1959), 173--182.
	\newblock DOI: \url{https://doi.org/10.1007/BF02410664}.
	
	\bibitem{Serrin1962}
	J. Serrin,
	On the interior regularity of weak solutions of the Navier--Stokes equations,
	\emph{Archive for Rational Mechanics and Analysis} \textbf{9} (1962), 187--195.
	\newblock DOI: \url{https://doi.org/10.1007/BF00253344}.
	

	\bibitem{Scheffer1976}
	V. Scheffer,
	Partial regularity of solutions to the Navier--Stokes equations,
	\emph{Pacific Journal of Mathematics} \textbf{66} (1976), no. 2, 535--552.
	\newblock DOI: \url{https://doi.org/10.2140/pjm.1976.66.535}.
	
	\bibitem{Scheffer1977}
	V. Scheffer,
	Hausdorff measure and the Navier--Stokes equations,
	\emph{Communications in Mathematical Physics} \textbf{55} (1977), 97--112.
	\newblock DOI: \url{https://doi.org/10.1007/BF01626512}.
	
	\bibitem{CKN1982}
	L. Caffarelli, R. Kohn, and L. Nirenberg,
	Partial regularity of suitable weak solutions of the Navier--Stokes equations,
	\emph{Communications on Pure and Applied Mathematics} \textbf{35} (1982), no. 6, 771--831.
	\newblock DOI: \url{https://doi.org/10.1002/cpa.3160350604}.
	
	\bibitem{SohrWahl1986}
	H. Sohr and W. von Wahl,
	On the regularity of the pressure of weak solutions of Navier--Stokes equations,
	\emph{Archiv der Mathematik} \textbf{46} (1986), no. 5, 428--439.
	\newblock DOI: \url{https://doi.org/10.1007/BF01210782}.
	
	\bibitem{Struwe1988}
	M. Struwe,
	On partial regularity results for the Navier--Stokes equations,
	\emph{Communications on Pure and Applied Mathematics} \textbf{41} (1988), no. 4, 437--458.
	\newblock DOI: \url{https://doi.org/10.1002/cpa.3160410404}.
	
	\bibitem{KozonoSohr1997}
	H. Kozono and H. Sohr,
	Regularity criterion on weak solutions to the Navier--Stokes equations,
	\emph{Advances in Differential Equations} \textbf{2} (1997), no. 4, 535--554.
	\newblock DOI: \url{https://doi.org/10.57262/ade/1366741147}.
	
	\bibitem{Lin1998}
	F.-H. Lin,
	A new proof of the Caffarelli--Kohn--Nirenberg theorem,
	\emph{Communications on Pure and Applied Mathematics} \textbf{51} (1998), no. 3, 241--257.
	\newblock DOI: \url{https://doi.org/10.1002/(SICI)1097-0312(199803)51:3<241::AID-CPA2>3.0.CO;2-A}.
	
	\bibitem{LadyzhenskayaSeregin1999}
	O. A. Ladyzhenskaya and G. A. Seregin,
	On partial regularity of suitable weak solutions to the three-dimensional Navier--Stokes equations,
	\emph{Journal of Mathematical Fluid Mechanics} \textbf{1} (1999), 356--387.
	\newblock DOI: \url{https://doi.org/10.1007/s000210050015}.
	
	
	\bibitem{ChoeLewis2000}
	H. J. Choe and J. L. Lewis,
	On the singular set in the Navier--Stokes equations,
	\emph{Journal of Functional Analysis} \textbf{175} (2000), no. 2, 348--369.
	\newblock DOI: \url{https://doi.org/10.1006/jfan.2000.3582}.
	
	\bibitem{SereginSverak2002}
	G. A. Seregin and V. \v{S}ver\'ak,
	Navier--Stokes equations with lower bounds on the pressure,
	\emph{Archive for Rational Mechanics and Analysis} \textbf{163} (2002), no. 1, 65--86.
	\newblock DOI: \url{https://doi.org/10.1007/s002050200199}.
	
	\bibitem{EscauriazaSereginSverak2003}
	L. Escauriaza, G. Seregin, and V. \v{S}ver\'ak,
	\(L_{3,\infty}\)-solutions of Navier--Stokes equations and backward uniqueness,
	\emph{Russian Mathematical Surveys} \textbf{58} (2003), no. 2, 211--250.
	\newblock DOI: \url{https://doi.org/10.1070/RM2003v058n02ABEH000609}.
	
	\bibitem{PenelPokorny2004}
	P. Penel and M. Pokorn\'y,
	Some new regularity criteria for the Navier--Stokes equations containing gradient of the velocity,
	\emph{Applications of Mathematics} \textbf{49} (2004), no. 5, 483--493.
	\newblock DOI: \url{https://doi.org/10.1023/B:APOM.0000048124.64244.7e}.
	
	\bibitem{KukavicaZiane2006}
	I. Kukavica and M. Ziane,
	One component regularity for the Navier--Stokes equations,
	\emph{Nonlinearity} \textbf{19} (2006), no. 2, 453--469.
	\newblock DOI: \url{https://doi.org/10.1088/0951-7715/19/2/012}.
	
	\bibitem{KukavicaZiane2007}
	I. Kukavica and M. Ziane,
	Navier--Stokes equations with regularity in one direction,
	\emph{Journal of Mathematical Physics} \textbf{48} (2007), no. 6, 065203, 10 pp.
	\newblock DOI: \url{https://doi.org/10.1063/1.2395919}.
	
	\bibitem{Seregin2007Morrey}
	G. A. Seregin,
	Estimates of suitable weak solutions to the Navier--Stokes equations in critical Morrey spaces,
	\emph{Journal of Mathematical Sciences} \textbf{143} (2007), no. 2, 2961--2968.
	\newblock DOI: \url{https://doi.org/10.1007/s10958-007-0178-2}.
	
	\bibitem{Seregin2007Local}
	G. A. Seregin,
	Local regularity for suitable weak solutions of the Navier--Stokes equations,
	\emph{Russian Mathematical Surveys} \textbf{62} (2007), no. 3, 595--614.
	\newblock DOI: \url{https://doi.org/10.1070/RM2007v062n03ABEH004415}.
	
	\bibitem{GustafsonKangTsai2007}
	S. Gustafson, K. Kang, and T.-P. Tsai,
	Interior regularity criteria for suitable weak solutions of the Navier--Stokes equations,
	\emph{Communications in Mathematical Physics} \textbf{273} (2007), 161--176.
	\newblock DOI: \url{https://doi.org/10.1007/s00220-007-0214-6}.
	
	\bibitem{Vasseur2007}
	A. Vasseur,
	A new proof of partial regularity of solutions to Navier--Stokes equations,
	\emph{Nonlinear Differential Equations and Applications} \textbf{14} (2007), 753--785.
	\newblock DOI: \url{https://doi.org/10.1007/s00030-007-6001-4}.
	
	\bibitem{ZhouPokorny2009}
	Y. Zhou and M. Pokorn\'y,
	On a regularity criterion for the Navier--Stokes equations involving gradient of one velocity component,
	\emph{Journal of Mathematical Physics} \textbf{50} (2009), no. 12, 123514, 11 pp.
	\newblock DOI: \url{https://doi.org/10.1063/1.3268589}.
	
	\bibitem{CaoTiti2011}
	C. Cao and E. S. Titi,
	Global regularity criterion for the 3D Navier--Stokes equations involving one entry of the velocity gradient tensor,
	\emph{Archive for Rational Mechanics and Analysis} \textbf{202} (2011), 919--932.
	\newblock DOI: \url{https://doi.org/10.1007/s00205-011-0439-6}.
	
	\bibitem{JiaSverak2014}
	H. Jia and V. \v{S}ver\'ak,
	Local-in-space estimates near initial time for weak solutions of the Navier--Stokes equations and forward self-similar solutions,
	\emph{Inventiones Mathematicae} \textbf{196} (2014), 233--265.
	\newblock DOI: \url{https://doi.org/10.1007/s00222-013-0468-x}.
	
	\bibitem{Seregin2015}
	G. A. Seregin,
	\emph{Lecture Notes on Regularity Theory for the Navier--Stokes Equations},
	World Scientific, Hackensack, NJ, 2015.
	\newblock DOI: \url{https://doi.org/10.1142/9314}.
	
	\bibitem{CheminZhang2016}
	J.-Y. Chemin and P. Zhang,
	On the critical one component regularity for 3-D Navier--Stokes system,
	\emph{Annales Scientifiques de l'\'Ecole Normale Sup\'erieure} \textbf{49} (2016), no. 1, 131--167.
	\newblock DOI: \url{https://doi.org/10.24033/asens.2278}.
	
	\bibitem{CheminZhangZhang2017}
	J.-Y. Chemin, P. Zhang, and Z. Zhang,
	On the critical one component regularity for 3-D Navier--Stokes system: general case,
	\emph{Archive for Rational Mechanics and Analysis} \textbf{224} (2017), no. 3, 871--905.
	\newblock DOI: \url{https://doi.org/10.1007/s00205-017-1089-0}.
	
	\bibitem{GuevaraPhuc2017}
	C. Guevara and N. C. Phuc,
	Local energy bounds and epsilon-regularity criteria for the 3D Navier--Stokes system,
	\emph{Calculus of Variations and Partial Differential Equations} \textbf{56} (2017), no. 3, article no. 68, 16 pp.
	\newblock DOI: \url{https://doi.org/10.1007/s00526-017-1151-7}.
	
	\bibitem{KukavicaRusinZiane2017}
	I. Kukavica, W. Rusin, and M. Ziane,
	An anisotropic partial regularity criterion for the Navier--Stokes equations,
	\emph{Journal of Mathematical Fluid Mechanics} \textbf{19} (2017), 123--133.
	\newblock DOI: \url{https://doi.org/10.1007/s00021-016-0278-1}.
	
	\bibitem{Wolf2017}
	J. Wolf,
	On the local pressure of the Navier--Stokes equations and related systems,
	\emph{Advances in Differential Equations} \textbf{22} (2017), no. 5/6, 305--338.
	\newblock DOI: \url{https://doi.org/10.57262/ade/1489802453}.
	
	\bibitem{HanLeiLiZhao2019}
	B. Han, Z. Lei, D. Li, and N. Zhao,
	Sharp one component regularity for Navier--Stokes,
	\emph{Archive for Rational Mechanics and Analysis} \textbf{231} (2019), 939--970.
	\newblock DOI: \url{https://doi.org/10.1007/s00205-018-1292-7}.
	
	\bibitem{BarkerPrange2021}
	T. Barker and C. Prange,
	Quantitative regularity for the Navier--Stokes equations via spatial concentration,
	\emph{Communications in Mathematical Physics} \textbf{385} (2021), 717--792.
	\newblock DOI: \url{https://doi.org/10.1007/s00220-021-04122-x}.
	
	\bibitem{KangNguyen2023}
	K. Kang and D. D. Nguyen,
	Local regularity criteria in terms of one velocity component for the Navier--Stokes equations,
	\emph{Journal of Mathematical Fluid Mechanics} \textbf{25} (2023), article no. 10, 15 pp.
	\newblock DOI: \url{https://doi.org/10.1007/s00021-022-00754-8}.
	
	\bibitem{AlbrittonBarkerPrange2023}
	D. Albritton, T. Barker, and C. Prange,
	Epsilon regularity for the Navier--Stokes equations via weak--strong uniqueness,
	\emph{Journal of Mathematical Fluid Mechanics} \textbf{25} (2023), article no. 49, 12 pp.
	\newblock DOI: \url{https://doi.org/10.1007/s00021-023-00780-0}.
	
\end{thebibliography}
\end{document}